\documentclass [11pt, reqno]{amsart}
\usepackage{mathrsfs}
\usepackage{mathtools}
\usepackage{amssymb}
\usepackage{enumerate}
\usepackage{verbatim}
\usepackage{comment}

\usepackage[
hypertexnames=false, colorlinks, citecolor=black, linkcolor=blue, urlcolor=red]{hyperref}

\normalbaselines						  %

\newenvironment{entry}
{\begin{list}{X}%
		{%
			\setlength{\labelwidth}{35pt}%
			\setlength{\leftmargin}{\labelwidth}
			\addtolength{\leftmargin}{\labelsep}%
			\setlength{\itemsep}{.4pc}%
		}%
	}
	{\end{list}} 

\newcommand{\T}{\mathbb{T}}

\newcommand{\Z}{{\mathbb  Z}}

\newcommand{\C}{{\mathbb  C}}
\newcommand{\R}{{\mathbb  R}}
\newcommand{\cH}{\mathcal{H}}
\newcommand{\cV}{\mathcal{V}}
\newcommand{\cE}{\mathcal{E}}

\numberwithin{equation}{section}

\newtheorem{thm}{Theorem}[section]
\newtheorem{lm}[thm]{Lemma}

\newtheorem{prop}[thm]{Proposition}
\newtheorem*{prop*}{Proposition}

\theoremstyle{definition}
\newtheorem{df}[thm]{Definition}
\newtheorem*{df*}{Definition}

\newcommand{\Ker}{\operatorname{Ker}}
\newcommand{\Ran}{\operatorname{Ran}}
\newcommand{\clos}{\operatorname{clos}}
\newcommand{\Arg}{\operatorname{Arg}}

\newcommand{\wt}{\widetilde}

\newcommand{\fdot}{\,\cdot\,}
\newcommand{\e}{\varepsilon}
\newcommand{\im}{\operatorname{Im}}
\newcommand{\re}{\operatorname{Re}}
\newcommand{\ran}{\operatorname{Ran}}
\newcommand{\spn}{\operatorname{span}}
\newcommand{\cspn}{\overline{\operatorname{span}}}
\newcommand{\rank}{\operatorname{rank}}
\newcommand{\dist}{\operatorname{dist}}

\newcommand{\cK}{\mathcal{K}}
\newcommand{\cM}{\mathcal{M}}
\newcommand{\ci}[1]{_{_{\scriptstyle #1}}}

\newcommand{\ut}[1]{^{\scriptstyle \text{\rm #1}}}
\newcommand{\utc}{\ut{c}}

\newcommand{\tup}[1]{\textup#1}

\theoremstyle{remark}
\newtheorem{rem}[thm]{Remark}
\newtheorem*{rem*}{Remark}
\newtheorem{defin}[thm]{Definition}

\renewcommand{\labelenumi}{\textup{(\roman{enumi})}}

\newcounter{vremennyj}

\newcommand\cond[1]{\setcounter{vremennyj}{\theenumi}\setcounter{enumi}{#1}\labelenumi\setcounter{enumi}{\thevremennyj}}

\begin{document}
\title[Inverse spectral problem]{A dynamical system approach 
	to the inverse spectral problem for Hankel operators: a model case}

\author{Zhehui Liang} 
\address{Department of Mathematics \\ Brown University \\ Providence, RI 02912 \\ USA}
\email{zhehui\_liang@brown.edu}

\author{Sergei Treil}
\thanks{Work of S.~Treil is 
	supported  in part by the National Science Foundation under the grants  
	DMS-1856719, DMS-2154321}
\address{Department of Mathematics \\ Brown University \\ Providence, RI 02912 \\ USA}
\email{treil@math.brown.edu}


\begin{abstract}
We present an alternative proof of the result by P.~G\'erard and S.~Grellier 
\cite{Ger-Grill_ISP_2014}, stating 
that
given two real sequences $(\lambda_n)_{n=1}^\infty$, $(\mu_n)_{n=1}^\infty$  satisfying the 
intertwining relations
\[
|\lambda_1| > |\mu_1| > |\lambda_2| > |\mu_2| > ...> |\lambda_n| > |\mu_n|>\ldots >0 , \qquad 
\lambda_n\to 0, 
\]
there exists a unique compact Hankel operator $\Gamma$ such that $\lambda_n$ are the (simple) 
eigenvalues of $\Gamma$ and $\mu_n$ are the simple eigenvalues of its truncation  $\Gamma_1$ 
obtained from $\Gamma$ by removing the first column. 

We use the dynamical systems approach originated in \cite{MPT-Acta}, and the proof is split into 
three independent parts.  The first one, which is a slight modification of a result in 
\cite{MPT-Acta} is an abstract operator-theoretic statement reducing the problem to the asymptotic 
stability os some operator. The second one is the proof of the asymptotic stability, which is 
usually the hardest part, but in our case of compact operators it is almost trivial. And the third 
part is an abstract version of the Borg's two spectra theorem, which is essentially a simple 
exercise in graduate complex 
analysis.  

\end{abstract}

\maketitle

\setcounter{section}{-1}
\section{Notation}
All operators act on or between Hilbert spaces, and we consider only separable Hilbert spaces. 
\begin{entry}
\item[$S$] the forward shift in $\ell^2$, $S(x_0, x_1, x_2, \ldots ) = (0, x_0, x_1, x_2, \ldots)$;
\item[$S^*$] adjoint of $S$, $S^* (x_0, x_1, x_2, \ldots ) = (x_1, x_2, x_3,  \ldots )$; 
\item[$|A|$] modulus of the operator $A$, $|A|:=(A^*A)^{1/2}$;
\item [$P_{\cE}$] the orthogonal projection  onto a subspace $\cE$.
\item [$a^*$] for $a\in\cH$, the symbol $a^*$ denotes the linear functional on  $\cH$, given by 
$a^*x =( x,a)\ci\cH $

\end{entry}

\section{Introduction and main results} 
A Hankel operator is a bounded linear operator in $\ell^2=\ell^2(\Z_+)$ with matrix whose entries 
depend on the sum of indices, 
\begin{align*}
\Gamma = \left(\gamma_{j+k}\right)_{j,k\ge0} \,.
\end{align*}
Denoting by $S$  the \emph{shift operator} in $\ell^2$,
\begin{align*}
S(x_0, x_1, x_2, \ldots ) &= (0, x_0, x_1, x_2, \ldots) \,, 
\intertext{and by $S^*$ its adjoint (the \emph{backward shift}), }
S^* (x_0, x_1, x_2, \ldots ) &= (x_1, x_2, x_3,  \ldots ) 
\end{align*}
we can see that an operator $\Gamma$ on $\ell^2$ is a Hankel operator if and only if 
\begin{align*}
\Gamma S = S^* \Gamma;  
\end{align*}
sometimes this formula is used for the definition of a Hankel operator.

In \cite{MPT-Acta} a complete description of self-adjoint operators unitarily equivalent to a 
Hankel 
operator was obtained. It was clear from the construction that such Hankel 
operator is not unique, and the question of finding additional spectral  conditions  providing 
uniqueness looks like a natural question. 

In a breakthrough series of papers \cite{Gerard 2010, Gerard 2012, Ger-Grill_ISP_2014, GG-Ast} 
P.~G\'erad 
and S.~Grellier  investigated the inverse 
problem for Hankel operators; their motivation come the study of the so-called cubic Szeg\"{o} 
equation, which is a completely integrable Hamiltonian system.  

One of their discoveries was that the spectral invariants of the \emph{pair} of Hankel operator 
$\Gamma$ and $\Gamma_1 = \Gamma S = S^*\Gamma$, \emph{completely} determine the operator $\Gamma$.

A simple illustration of that principle, is Theorem \ref{t:main 02} below, proved in 
\cite{Ger-Grill_ISP_2014}. 

In this paper we present a simple proof of this result. 
\begin{thm}
\label{t:main 02}  Given sequences $((\lambda_n)_{n=1}^{\infty}$  and  $(\mu_n)_{n=1}^{\infty}$ of 
non-zero real numbers, satisfying intertwining relations
\begin{align}
\label{intertwine 01}
|\lambda_{1}|>|\mu_{1}|>|\lambda_{2}|>|\mu_{2}|>....>|\lambda_{n}|>|\mu_{n}|>...>0, 
\end{align}
and such that $\lim\limits_{n \to \infty}\lambda_n= 0$, $\lim\limits_{n \to \infty}\mu_n= 0$ , 
there exists a unique self-adjoint compact Hankel operator $\Gamma$  such that non-zero eigenvalues 
of $\Gamma$ and $\Gamma S$ are simple, and coincide with $(\lambda_n)_{n=1}^{\infty}$ and  
$(\mu_n)_{n=1}^{\infty}$ respectively.  

  Moreover, $\ker \Gamma =\{ 0\}$ if and only if: 
\begin{align}
\label{trivial kernel condition 1}
    \sum\limits_{j=1}^{\infty} \Bigg( 1-\frac{\mu_{j}^{2}}{\lambda_{j}^{2}} \Bigg) = \infty \\
\label{trivial kernel condition 2}
\sum_{j=1}^\infty \left( \frac{\mu_j^2}{\lambda_{j+1}^2} -1 \right) =\infty 
\end{align}
\end{thm}

The proof presented in \cite{Ger-Grill_ISP_2014} was based on a deep analysis of Hankel operators. 
 
In this paper we present a simple proof consisting of three essentially separate parts. The first 
one is a simple operator-theoretic statement (which was essentially proved in \cite{MPT-Acta}); it 
is presented in Section \ref{s:plan}  below. The second part, which was  the hardest part in 
\cite{MPT-Acta}, is proving the asymptotic stability of some operator. However  in our case of 
compact operator it can be obtained essentially for free. This part is discussed in Section 
\ref{asymptotic stability}. 

 The third part is an abstract version of the Borg's two spectra theorem in \cite{Borg_Acta_1946}, 
 which 
is just  an exercise in graduate complex analysis. 

\begin{rem}
The above theorem holds for the finite rank case as well. In this case we have finite sequences 
$(\lambda_n)_{n=1}^N$ and $(\mu_n)_{n=1}^N$ satisfying the intertwining relations 
\eqref{intertwine 01}, $\lambda\ci{N} \ne 0$  (but $\mu\ci{N}$ can be $0$). 

Then there exists a unique finite rank Hankel operator $\Gamma$ such that non-zero eigenvalues of 
$\Gamma$ are simple and coincide with $(\lambda_n)_{n=1}^N$, and the non-zero eigenvalues of 
$\Gamma S$ are also simple and coincide with non-zero members of $(\mu_n)_{n=1}^N$. 

The kernel of $\Gamma$ is always non-trivial, which agrees with Theorem \ref{t:main 02}, because 
for finite sequences conditions \eqref{trivial kernel condition 1}, \eqref{trivial kernel condition 
2} always fail. 

The proof of the finite rank case is simpler, since the asymptotic stability and the abstract 
Borg's theorem are trivial in this case. 
\end{rem}

\section{An abstract inverse problem theorem}
\label{s:plan}
\subsection{Plan of the game}

Let $\Gamma$ be a self-adjoint Hankel operator. Define $\Gamma_1= \Gamma S = S^*\Gamma$ (which is 
also a self-adjoint Hankel operator). Let $(e_n)_{n=0}^\infty$ be the standard basis in $\ell^2$. 
Using the fact that $SS^* = I-(\fdot, e_0)e_0 = I- e_0 e_0^*$  we can write  
\begin{align}
\label{e: Gamma^2 - uu*}
\Gamma_1^2 = \Gamma S^* S \Gamma = \Gamma (I- e_0 e_0^*) \Gamma = \Gamma^2 - uu^*, 
\end{align}
where $u:=\Gamma e_0$.

Now let us try to go the opposite direction. Suppose that we are given two self-adjoint operators 
$R$, $R_1$ on a Hilbert space $H$,   and we want to find a Hankel operator $\Gamma$, such that the 
Hankel operators $\Gamma$ and $\Gamma_1= \Gamma S = S^*\Gamma$ are unitarily equivalent to $R$ and 
$R_1$ respectively.   We can see from \eqref{e: Gamma^2 - uu*}  that $R$ and $R_1$ should satisfy 
the relations 
\begin{align}
\label{R^2 01}
    R^2-R_{1}^{2}=pp^*, 
\end{align}
for some  $p \in H$.

To solve the inverse problem we want to find an operator $\Sigma^*$, (hopefully unitarily 
equivalent to the backward shift $S^*$)  such that $R_1=\Sigma^* R$.
The tool to find $\Sigma^*$ is the following simple lemma.

\begin{lm}[Douglas Lemma]
\label{l:Douglas}
Let $A$ and $B$ be bounded operators in a Hilbert space $\cH$ such that 
\[
\|Bh\|\le \|Ah\| \qquad \forall h\in\cH, 
\]
or, equivalently, $B^*B\le A^*A$. 

Then there exists a contraction $T$ (i.e.~$\|T\|\le 1$) such that $B=TA$. 

Moreover, if $A$ has dense range, the operator $T$ is unique. 
\end{lm}

\begin{rem}
\label{r: Douglas} If $\ker A=\{0\}$ and $A$ has dense range, then the (unbounded) operator 
$A^{-1}$ is defined on the dense set $\ran A$, and the operator $T$ is given on this dense set by 
$T=BA^{-1}$. The adjoint $T^*$ is given by $T^*=(A^{-1})^* B^*$; it is not hard to show that $\ran 
B^*$ is in the domain of $(A^{-1})^*$, so it is well defined on all space.  
\end{rem}

To avoid non-uniqueness in finding $\Sigma^*$ it is convenient to deal only with the \emph{core} of 
the operator $\Gamma$, i.e.~with the operator $\Gamma\utc:=\Gamma |_{(\ker\Gamma)^\perp}$.  So, 
given self-adjoint operators $R$ and $R_1$ satisfying \eqref{R^2 01}, 
$\ker R=\{0\}$, we want to find Hankel operator $\Gamma$ such that the pair $R$, $R_1$ is unitarily 
equivalent to the pair $\Gamma\utc=\Gamma |_{(\ker\Gamma)^\perp}$, $\Gamma_1\utc=S^*\Gamma 
|_{(\ker\Gamma)^\perp}$. In this case, by the above Lemma \ref{l:Douglas} there exists a unique 
contraction   $\Sigma^*$ such that $R_1=\Sigma^* R$ (note that $\Sigma^*=R_1R^{-1}$, see Remark 
\ref{r: Douglas}). Clearly, if the pair $R$, $R_1$ is unitarily equivalent to the pair 
$\Gamma\utc$, $\Gamma_1\utc$, then $\Sigma^*$ 
should be  unitarily equivalent to $S^*$ restricted to $(\ker \Gamma)^{\perp}$, meaning that 
\[
\Sigma^* = V^*S^*|_{(\Ker \Gamma)^\perp}V.
\]
for some unitary operator $V$. 

In this situation the vector $p$ from \eqref{R^2 01} should be $p=Rq$, where $\|q\|\le 1$; think of 
$q$  as of  being
\begin{align*}
    q = V^*P_{(\ker\Gamma)^\perp} e_0 .
\end{align*}

Using the fact that $R_1=\Sigma^*R = R\Sigma$ we can rewrite \eqref{R^2 01} as 
\begin{align*}
R(I- \Sigma\Sigma^*)R = R qq^* R
\end{align*}
which implies that 
\begin{align}
\label{defect T}
    I -\Sigma\Sigma^* &= q q^* = (\cdot , q) q
\end{align}

So,  we arrive at the following setup: $R$ and $R_1$ are self-adjoint operators, $\ker R=\{0\}$, 
\begin{align}
\label{e: R^2 - R_1^2}
    R^2-R_{1}^{2}=pp^*, \qquad p=Rq, \quad \|q\|\le 1, 
\end{align}
 and $\Sigma^*$ is the unique contraction satisfying $R_1 = \Sigma^* R$; note that in this case 
 $\Sigma^*$ satisfies \eqref{defect T}. 

\begin{rem}
\label{r: Sigma*}
The operator $\Sigma^*$ is defined on a dense set as $\Sigma^*= R_1R^{-1}$, see Remark \ref{r: 
Douglas}
\end{rem}
\begin{defin}
We say an operator $T$ is asymptotically stable iff $T^n\to 0$  in the strong operator topology as 
$n\to \infty$, i.e. iff for for all $ x\in \cH$
\begin{align*}
\lim\limits_{n \mapsto \infty}\|T^{n}x\| \rightarrow 0 .
\end{align*}
\end{defin}

\begin{prop}
\label{uniqueHankel}
If the operator $\Sigma^*$ introduced above is asymptotically stable, then there exist a unique 
Hankel 
operator $\Gamma$ such that the pair $\Gamma |_{(\Ker \Gamma)^\perp}$, $\Gamma S|_{(\Ker 
\Gamma)^\perp}$ is unitary equivalent to $R$, $R_1$, i.e.~that there exists  a unitary operator $V 
: \cH \to (\Ker\Gamma)^\perp$ such that
\begin{align}
\label{Gamma equiv R}
\Gamma |_{(\Ker \Gamma)^\perp} & = VRV^*,  \\
\label{Gamma_1 equiv R_1}
\Gamma S|_{(\Ker \Gamma)^\perp} &= VR_{1}V^* . 
\intertext{Moreover, multiplying $V$ by a unimodular constant one can always get that}  
\label{e:Vq}
P\ci{(\ker\Gamma)^\perp} e_0  & = V q ,  \\
\label{e:Vp}
\Gamma e_0 & = Vp .
\end{align}
Finally, $\ker \Gamma =\{0\}$ if and only if  $\|q\| =1$ and $q \notin \Ran R$.
\end{prop}

\begin{rem}
\label{r:uniq V}
The unitary operator $V$ is clearly not unique, since replacing $V$ by $\alpha V$, $
\lvert \alpha \rvert = 1$ does not change \eqref{Gamma equiv R} and \eqref{Gamma_1 equiv R_1}. 
However, the multiplication by a unimodular constant is the only degree of freedom for $V$; we will 
see in the proof of Proposition \ref{uniqueHankel} that a unitary operator $V$ satisfying 
\eqref{Gamma equiv R},  \eqref{Gamma_1 equiv R_1} and one of the identities \eqref{e:Vq}, 
\eqref{e:Vp} is unique. 
\end{rem}
\begin{proof}[Proof of Proposition \ref{uniqueHankel}]
Treating \eqref{defect T} as an identity for quadratic forms and substituting $x\in \cH$ into it we get 
\[
\|x\|^2 - \|\Sigma^*x\|^2 = |(x,q)|^2. 
\]
Applying this identity to $(\Sigma^*)^kx$ we get 
\[
\|(\Sigma^*)^k x\|^2 - \|(\Sigma^*)^{k+1}x\|^2 = |((\Sigma^*)^k x,q)|^2, 
\]
so taking the sum we get
\[
\|x\|^2 - \|(\Sigma^*)^{n+1}x\|^2 = \sum_{k=0}^n |((\Sigma^*)^k x,q)|^2 . 
\]
Taking the limit as $n\to\infty$ and using the asymptotic stability of $\Sigma^*$ we see that
\begin{align*}
\|x\|^2 = \sum_{k=0}^\infty |((\Sigma^*)^k x,q)|^2, 
\end{align*}
which means that the operator $\cV:\cH\to \ell^2$, 
\begin{align}
\label{e: def V}
\cV x :=  \left( (x,q), (\Sigma^*x, q), ((\Sigma^*)^2x, q), \ldots   \right)  =\left( ((\Sigma^*)^k 
x, q) \right)_{k=0}^\infty
\end{align}
is an isometry. 

We can see that 
\[
\cV \Sigma^*x = \left(  (\Sigma^*x, q), ((\Sigma^*)^2x, q), ((\Sigma^*)^3 x, q),  \ldots   \right) 
= S^* V x, 
\]
i.e.~$\cV\Sigma^* =S^*\cV$,  
so $\Sigma^*$ is unitarily equivalent to either $S^*$ (if $\ran \cV = \ell^2$) or to the 
restriction of $S^*$ to $S^*$-invariant subspace  $\ran \cV\subset\ell^2$ (if $\ran \cV\ne\ell^2$). 

Denoting by $\Sigma$ the adjoint of $\Sigma^*$ and taking the adjoint of $\cV\Sigma^* =S^*\cV$ we 
get  
$\Sigma\cV=\cV S$. 

Define $\Gamma:= \cV R \cV^*$.  Then 
\begin{align}
\label{e: Gamma S}
\Gamma S = \cV R \cV^* S =  \cV R \Sigma \cV = \cV \Sigma^* R \cV  = S^* \cV R\cV^* = S^*\Gamma;
\end{align}
(here in the second equality we used $\cV\Sigma^* =S^*\cV$, and in the forth one $\cV\Sigma^* 
=S^*\cV$), so $\Gamma$ is a Hankel operator. 

We can also see from \eqref{e: Gamma S} that $\Gamma S = \cV\Sigma^* R\cV^*= \cV R_1 \cV^*$. 

Let $V$ be the operator $\cV$ with the target space restricted to $\Ran \cV$, so $V : \cH\to \Ran 
\cV$ is a unitary operator. Since $\Ker R=\{0\}$, the identity $\Gamma = \cV R \cV^*$ implies that 
$\ker \Gamma=\{0\}$, so identities $\Gamma = \cV R \cV^*$, $\Gamma_1 = \cV R_1 \cV^*$ translate to 
\eqref{Gamma equiv R}, \eqref{Gamma_1 equiv R_1}.  

Let us now discuss  identities \eqref{e:Vq}, \eqref{e:Vp}. We can see from the definition of $\cV$ 
that  
\begin{align*}
\left( x, \cV^* e_0 \right)\ci\cH = (\cV x, e_0)\ci{\ell^2} = (x,q)\ci\cH, 
\end{align*}
so $\cV^*e_0 =q$, which is equivalent to \eqref{e:Vq}. Next, 
\begin{align*}
u=\Gamma e_0 = \cV R\cV^* e_0 = \cV R q= \cV p, 
\end{align*}
which is exactly \eqref{e:Vp}. Thus for the operator $V$ constructed above, identities 
\eqref{e:Vq}, \eqref{e:Vp} are satisfied. 

Let us  discuss uniqueness. Suppose the identities 
\eqref{Gamma equiv R}, \eqref{Gamma_1 equiv R_1} hold for some unitary operator $V: \cH \to 
\clos 
\Ran \Gamma = (\Ker \Gamma)^\perp \subset \ell^2$, not necessarily the one constructed above. 

Since for a Hankel operator $\Ker \Gamma$ is always $S$-invariant, the subspace 
$(\Ker\Gamma)^\perp$ is $S^*$-invariant, so the restriction $S^*|_{(\Ker \Gamma)^\perp}$ is well 
defined. The identities \eqref{Gamma equiv R}, \eqref{Gamma_1 equiv R_1} with $V=\wt V$ and the 
definition of $\Sigma^*$ then imply that 
\begin{align}
\label{S* equiv T}
  S^*|_{(\Ker \Gamma)^\perp}  =  V \Sigma^*  V^* . 
\end{align}

As it was discussed in Section \ref{s:plan}, see \eqref{e: Gamma^2 - uu*}, 
\begin{align*}
(\Gamma S)^2 = \Gamma^2 - uu^*, 
\end{align*}
where 
\begin{align*}
u:= \Gamma e_0 = \Gamma P_{(\Ker\Gamma)^\perp}e_0. 
\end{align*}
Then identities \eqref{Gamma equiv R}, \eqref{Gamma_1 equiv R_1} imply that 
\[
R_1^2 = R^2 - \wt p \wt p^*, \qquad \text{where}\quad  \wt p= R \wt q, \quad \wt q=  V^* 
P_{(\Ker\Gamma)^\perp}e_0 . 
\]
Comparing this with $R_1^2 = R^2 -  p  p^*$ we conclude that $\wt p=\alpha  p$, $\wt q= \alpha  q$, 
$|\alpha| =1$, so multiplying $V$ by $\alpha$ we get the identities \eqref{e:Vq},  
\eqref{e:Vp} (without spoiling \eqref{Gamma equiv R}, \eqref{Gamma_1 equiv R_1}). 

Computing the coefficients $\gamma_k$ of the operator $\Gamma$ we write 
\begin{align*}
\gamma_k = (\Gamma e_0, S^k e_0) & = \left((S^*)^k \Gamma e_0, e_0\right)  = 
\left((S^*\big|_{(\Ker\Gamma)^\perp})^k \Gamma P_{(\Ker\Gamma)^\perp} e_0, P_{(\Ker\Gamma)^\perp} 
e_0\right)
\\ & = \left( (\Sigma^*)^k R \tilde q, \tilde q  \right) = \left( (\Sigma^*)^k \wt p, \wt q 
\right) = 
\left( (\Sigma^*)^k  p,  q \right), 
\end{align*}
meaning that the coefficients $\alpha_k$ do not depend on $V$. So, uniqueness of the Hankel 
operator $\Gamma$ is proved. 

Finally, let us discuss the kernel of $\Gamma$. As we discussed above, $\Ker \Gamma$ is trivial if 
and only if  the operator $V$ defined by \eqref{e: def V} satisfies $\Ran V =\ell^2$. In this case 
$\Sigma^*$ is unitarily  equivalent to $S^*$, or, equivalently, $\Sigma$ is unitarily equivalent to 
$S$.  

So, let $\Ker\Gamma=\{0\}$, so  $\Sigma^*$ is unitarily equivalent to $S^*$. We know that 
\[
I- SS^* = e_0e_0^*, \qquad I - \Sigma \Sigma^* = qq^*, 
\]
so by the unitary equivalence $\|q\|=\|e_0\|=1$, and $\Sigma^*q=0$. If $q\in\Ran R$, i.e.~$q=Rf$, 
then 
\begin{align*}
R\Sigma f = \Sigma^* Rf = \Sigma^* q =0, 
\end{align*}
and since $\Ker R =\{0\}$ , we conclude that $\Sigma^*f=0$. But if $\Ker\Gamma =\{0\}$, then 
$\Sigma$ is 
an isometry, which contradicts $\Sigma f=0$. 
So $q\notin \Ran R$. 

Let now $\|q\|=1$ and $q\notin \Ran R$. We know that 
\begin{align*}
\Sigma^* R^2 \Sigma = R_1^2 = R \Sigma\Sigma^* R = R(I-qq^*)R , 
\end{align*}
and that $\Ker (I-qq^*) =\spn\{q\}$. Since $q\notin \Ran R$, we see that $\Ker R(I-qq^*)R =\{0\}$, 
so $\Ker \Sigma =\{0\}$.  

Applying \eqref{defect T} to vector $q$ we get that 
\begin{align*}
q - \Sigma\Sigma^*q = q, 
\end{align*}
and since $\Ker \Sigma=\{0\}$ we see that $\Sigma^*q=0$. 

Left and right multiplying \eqref{defect T} by $\Sigma^*$ and $\Sigma$ respectively, we get that 
\begin{align*}
\Sigma^*\Sigma  -\Sigma^*\Sigma\Sigma^*\Sigma = \Sigma^*qq^*\Sigma , 
\end{align*}
and since $\Sigma^*q = 0$,  we have $\Sigma\Sigma^* = (\Sigma\Sigma^*)^2$, which means 
$\Sigma\Sigma^*$ is an orthogonal projection. 

Since $\Ker \Sigma =\{0\}$, we conclude that $\Sigma^*\Sigma=I$, i.e.~that $\Sigma$ is an isometry. 
Since $\Sigma^*$ is asymptotically stable, $\Sigma^*$ (and so $\Sigma$) has no unitary part. The 
identity \eqref{defect T} implies that $\rank (I-\Sigma\Sigma^*)=1$, so $\Sigma$ is unitarily 
equivalent to the 
shift operator $S$, i.e.~that 
\[
S = V \Sigma V^*
\]
for some unitary operator $V:\cH\to \ell^2$.  Defining $\Gamma= VRV^*$, $\Gamma_1 = VR_1V^*$, we 
can see that 
\[
\Gamma_1 = \Gamma S = S^*\Gamma, 
\]
so $\Gamma$ is indeed the Hankel operator with trivial kernel, satisfying the conditions 
\eqref{Gamma equiv R}, \eqref{Gamma_1 equiv R_1}  (we already proved that such Hankel operator is 
unique). 
\end{proof}

%

\section{Asymptotic stability}
\label{asymptotic stability}


The hard part of solving the inverse problem for Hankel operator is usually the proof of asymptotic 
stability of operator $T$. However, under the compact operator case we will get asymptotic 
stability for free. 

Let us recall the setup. We had compact self-adjoint operators $R$ and $R_1$, $\ker R=\{0\}$,  
satisfying $R_1^2 = R^2-pp^*$, where $\|R^{-1}p\|\le 1$,  with $\{\lambda_k\}_{k\ge 1}$ and 
$\{\mu_k\}_{k\ge 1}$ being the non-zero eigenvalues of $R$ and $R_1$ respectively. We also assume 
that the eigenvalues satisfy the intertwining relations \eqref{intertwine 01}.

\begin{prop}
\label{p:asy stability}
Under the above assumptions the vector $p$ is cyclic for $R^2$ and the operator $\Sigma^*$ is 
asymptotically stable. 
\end{prop}

The cyclicity of $p$ is easy. Indeed, if $p$ is not cyclic for $R^2$, then projection of $p$ onto 
some eigenspace $\ker (R^2-\lambda_k I)$ is zero, so operators $R^2$ and $R_1^2$ coincide on their 
common eigenspace $\ker (R^2-\lambda_k I)$, which means that $|\lambda_k|=|\mu_k|$. 

\subsection{Preparation}
\label{preparation work}
To prove the asymptotic stability of the  contraction $\Sigma^*$ we will use  the following simple 
lemma, which is a slight modification of \cite[lemma 3.2]{MPT-Acta}.

 \begin{lm}
 \label{existence of A}
 Let $\|T\|\le 1$, and let $K$ be a compact operator with a dense range. 
 Assume that  an operator $A$ satisfies
 \begin{align}
 \label{TK=KA}
T K = K A. 
 \end{align}

If $A$ is weakly asymptotically stable, meaning that $A^n\to 0$ in the weak operator topology 
(W.O.T) as  $n\to\infty$, then $T$ is asymptotically stable. 
 \end{lm}
 
\begin{proof} 
Iterating \eqref{TK=KA} we get that $T^n K = KA^n$, $n\ge 1$. Take $x\in\cH$. Since $A^n\to0$ in 
W.O.T. and $K$ is compact, we have that $\|KA^n x\|\to 0$. 

So $\lim_{n\to\infty}\|T^n y\| = 0$ for all $y\in\Ran K$. Thus, we have strong convergence on a 
dense set, and since $\|T^n\|\le 1$, we conclude (by $\e/3$-Theorem) that $T^n\to 0$ in the strong 
operator topology. 
\end{proof}

    Recall, that for an operator $R$ (in a Hilbert space) its \emph{modulus} $|R|$ is defined as 
    $|R|:= (R^*R)^{1/2}$
\begin{lm}
\label{contraction A}
For the operators $R$ and $\Sigma^*$ from Section \ref{s:plan} 
there exists a unique contraction $A$, such that:  
 \begin{equation*}
\Sigma^*|R|^{1/2}=|R|^{1/2}A
 \end{equation*}
\end{lm}
\begin{proof}
The inequality $\Sigma\Sigma^*\le I$ implies that 
\begin{align*}
|R|^2 = R^2 \ge R\Sigma\Sigma^*R = \Sigma^*R^2\Sigma = \Sigma^*|R|^2 \Sigma \ge 
\Sigma^*|R|\Sigma\Sigma^*|R|\Sigma = (\Sigma^*|R|\Sigma)^2. 
\end{align*}  
Recall, that the L\"{o}wner--Heinz inequality states that for self-adjoint operators $W$, $W_1$ the 
inequality $W\ge W_1\ge 0$ imply that $W^\alpha\ge W_1^\alpha$ for all $\alpha\in(0,1)$. Applying 
this inequality with $\alpha=1/2$ to operators $|R|^2$ and $( \Sigma^*|R|\Sigma)^2$, we conclude 
that 
\[
|R|  \ge \Sigma^*|R|\Sigma, 
\]
or, equivalently, 
\[
\| |R|^{1/2}x\| \ge \| |R|^{1/2} \Sigma x\| \qquad \forall x\in\cH .
\]
By Lemma \ref{l:Douglas} there exists a unique contraction, denote it to be $A^*$, satisfies: 
\[
A^*|R|^{1/2} = |R|^{1/2} \Sigma .
\]
Taking the adjoint, we get the conclusion of the lemma. 
\end{proof}

The operator $A$ constructed in the above Lemma \ref{contraction A} satisfies the identity 
\eqref{TK=KA} with $K=|R|^{1/2}$. Since $|R|^{1/2}$ is compact,   Lemma \ref{existence of A} says 
that the weak asymptotic stability of $A$ implies the asymptotic stability of $T$. 

\subsection{Weak asymptotic stability of \texorpdfstring{$A$}{A}} 

We will show below in Section \ref{s: structure A} that under our assumptions the operator $A$ is a 
strict contraction, meaning that $\|Ax\|<\|x\|$ for all $x\ne 0$. 

\begin{lm}
\label{pure contraction to weakly asyptotic stable}
Let $A:H\to H$ be a strict contraction. 
Then $A^n \to 0$ in the weak operator topology (WOT) as $n\to\infty$.
\end{lm}

\begin{proof}
First we notice that the assumption that $A$ is a strict contraction implies that $A$ is completely 
non unitary, meaning that there is no reducing subspace of $A$ on which $A$ acts unitarily. But 
every completely non-unitary contraction  admits the \emph{functional model}, i.e. it is unitaryly 
equivalent to the model operator $\cM_\theta$ on the model space $\cK_\theta$, where $\theta$ is 
the so-called \emph{characteristic function} of $A$. Without going into details, which are not 
important for our purposes, we just mention that the model space $\cK_\theta$ is a subspace of a 
vector-valued space $L^2(E)=L^2(\T,  m;E)$ of square integrable (with respect to the normalized 
Lebesgue measure $m$ on $\T$) functions with values in an auxiliary Hilbert space $E$. 
The model operator $\cM_\theta$, to which $A$ is unitarily equivalent,  is just the 
\emph{compression} of the multiplication operator$M_z$ by the independent variable $z$
\begin{align*}
\cM_\theta f = P\ci{\cK_\theta} M_z f, \qquad f\in \cK_\theta ;
\end{align*}
recall that  the  multiplication operator $M_z$ is defined by $M_z f(z) = zf(z)$, $z\in\T$. 

What is also essential for our purposes, is that the multiplication operator $M_z$ is the 
\emph{dilation} of the model operator $\cM_\theta$, i.e.~that for all $n\ge 1$
\begin{align*}
\cM_\theta^n f = P\ci{\cK_\theta}M_z^n f, \qquad f\in \cK_\theta. 
\end{align*}

Since trivially $M_z^n \to 0$ in the weak operator topology of $B(L^2(\T, m;E))$ as $n\to+\infty$, 
we conclude that $\cM_\theta^n\to 0$ as $n\to +\infty$ in  the weak operator topology of 
$B(\cK_\theta)$, and so $A^n\to 0$ in the weak operator topology as well. 
\end{proof}

\subsection{\texorpdfstring{$A$}{A} is a strict contraction}
\label{s: structure A}
To prove the weak asymptotic stability of $A$ we need to investigate its structure in more 
detail.

We know that $|R_1|^2 = |R|^2- pp^* \le|R|^2$. By the L\"{o}wner--Heinz inequality with 
$\alpha=1/2$ we have that $|R_1|\le|R|$, so
by Lemma \ref{l:Douglas} there exists a unique contraction $Q$ such that: 
\begin{align}
 |R_{1}|^{1/2} &= Q|R|^{1/2}
 \label{contraction Q}
\end{align}  
Let us write the polar decompositions of $R$ and $R_1$ 
\begin{align*}
  R = J |R|\qquad R_1 = J_1 |R_1| . 
\end{align*}
Since $R$ and $R_1$ are self-adjoint, the unitary operators $J$ and $J_1$ are also self-adjoint and 
commute with $|R|$ and $|R_1|$ respectively.

 The following simple Lemma (see  \cite[Lemma 3.5]{MPT-Acta}) gives the structure of the operator 
 $A$.   
\begin{lm}
\label{structure of A}
\begin{equation*}
  A=Q^* J_1 QJ
\end{equation*}  
\end{lm}

\begin{proof}
Note (see Remark \ref{r: Douglas}) that $Q=|R_1|^{1/2} |R|^{-1/2}$.  Using  Remark \ref{r: 
Douglas} again we can write
\begin{align*}
A= |R|^{-1/2} \Sigma^* |R|^{1/2} &= |R|^{-1/2} R_1 R^{-1} |R|^{1/2}  = |R|^{-1/2} |R_1| J_1 J 
|R|^{-1/2}  \\
& = |R|^{-1/2} |R_1|^{1/2} J_1 R_1|^{1/2} |R|^{-1/2} J = Q^*J_1 QJ ;
\end{align*}
here we used the fact that $J$ and $J_1$ commute with $|R|$ and $|R_1|$ respectively. 
\end{proof}

In addition, the following lemma, see \cite[Lemma 3.6]{MPT-Acta} gives the structure of $Q$
\begin{lm}
\label{structure of Q}
Let $\cH_{0}$ be the smallest invariant subspace of $|R|$ that contains $|R|^{-1}p=Jq$ (recall that 
$q$ is defined by $p=Rq$ with $\| q \| \leq 1$). Then $Q$ has the following block structure in the 
decomposition $\cH=\cH_0\oplus\cH_0^\perp$:
\begin{align}
\label{e:struct Q}
Q=\begin{pmatrix}
Q_{0} & 0 \\
0 &  I  \end{pmatrix}
\end{align}  
 where  $Q_{0}$ is a pure contraction (i.e.~$\| Q_0h\|<\|h\|$ for all $h\ne0$).  
\end{lm}

\begin{proof}  This proof essentially repeats  the proof from \cite{MPT-Acta}, we present it here 
only for the reader's convenience. 

We know that 
\begin{align}
\label{e:R_1^2 02}
R_1^2 &= R^2 -pp^*,
\end{align}
so $R^2$ coincides with $R_1^2$ on $\cH_0^\perp$. 

One can easily see that  $\cH_0$ is a reducing subspace for $R^2$ and for $R_1^2$, and these 
operators in the decomposition $\cH = \cH_0 \oplus \cH_0^\perp$ are block diagonal. 

It is also easy to see that $R^{1/2}$ and $R_1^{1/2}$ coincide on $\cH_0^{\perp}$, 
which is a reducing space for both operators, so $Q$ has the form \eqref{e:struct Q}, 
and only need to show that $Q_0$ is a strict contraction.

Using \eqref{e:R_1^2 02} and the identity $|R_1|^{1/2} = Q|R|^{1/2} = |R|^{1/2} Q^*$, we can write
\begin{align*}
R^2 -pp^* &= R_1^2  
         = |R|^{1/2}Q^*Q|R|Q^*Q|R|^{1/2} . 
\end{align*}
Recalling that $p= Rq = J|R| q = |R| J q$, we can rewrite the above identity as 
\begin{align*}
|R|^{1/2} \left(   |R|-(|R|^{1/2}Jq)(|R|^{1/2}Jq)^*   \right)  |R|^{1/2}   = 
|R|^{1/2}Q^*Q|R|Q^*Q|R|^{1/2} 
\end{align*}
which, because $\Ker R=\{0\}$, implies that
\begin{align}
\label{equation 1 of Q}
Q^*Q|R|Q^*Q = |R|-(|R|^{1/2}Jq)(|R|^{1/2}Jq)^*  .
\end{align}
Applying both sides to $x$, and taking the inner product with $x$, we get
\begin{align}
\label{equation 2 of Q}
(|R|Q^*Qx,Q^*Qx) = (|R|x,x)-|(x, |R|^{1/2}Jq)|^2 .
\end{align}

Now, take $x$ such that $\|Qx\|=\|x\|$. Since $\|Q\|\le 1$, this happens if and only if $x = 
Q^*Qx$.  The equation \eqref{equation 2 of Q} can be rewritten in this case as 
\begin{align*}
(|R| x, x) = (|R|x,x)-|(x, |R|^{1/2}Jq)|^2 , 
\end{align*}
which implies that $x \perp |R|^{1/2}Jq$.  Applying equation \eqref{equation 1 of Q} to such $x$, 
and 
using again the fact that $Q^*Qx=x$,  we get that 
\[
Q^*Q|R|x = |R|x .
\]
Hence set $\cH_1:= \left\{h\in \cH : h \in \cH, \|Qh\| = \|h\| \right\} =\Ker(I-Q^*Q)$ is an 
invariant subspace for $|R|$ (and so for $|R|^{1/2}$), which is orthogonal to $J|R|^{1/2}q$. 
Therefore
\begin{align*}
\cH_1 \perp \cspn\{|R|^{n/2} |R|^{1/2} J q : n\ge0\} & = \cspn\{|R|^{n/2} p : n\ge2\} 
\\ &\supset\cspn\{|R|^{n} p : n\ge1\} = \cH_0 ; 
\end{align*}
in the last equality we used the fact that  $|R|p$ is also cyclic for $|R|\big|_{\cH_0}$. 
Thus 
 $Q_0=Q|_{\cH_0}$ is a strict contraction, and the lemma is proved.
\end{proof}


Returning to our situation, recall that the operator $R_1^2$ is a rank one perturbation of $R^2$, 
\[
R_1^2 = R^2-pp^*, 
\]
and that $R^2$ and $R_1^2$ have simple eigenvalues $\lambda_k^2$ and $\mu_k^2$ respectively. The 
strict intertwining relations \eqref{intertwine 01} imply that for all $k$
\begin{align*}
P\ci{\Ker (R^2-\lambda_k^2 I)} p \ne 0
\end{align*}
(because if $P\ci{\Ker (R^2-\lambda_k^2 I)} p = 0$ then $\Ker (R^2-\lambda_k^2 I) = \Ker 
(R_1^2-\lambda_k^2 I)$, so $\mu_k=\lambda_k$), so $p$ is a cyclic vector for $|R|$. Therefore in 
out case $\cH_0=\cH$, so $Q$ is a pure contraction. Since $A=Q^*J_1QJ$, the operator $A$ is also a 
pure contraction. \hfill\qed  

So far we have shown that the contraction $T$ is asymptotically stable. This finishes our proof for 
the existence and uniqueness of unitary equivalent Hankel operator in Prop \ref{uniqueHankel}.

\section{Abstract Borg's theorem }

Let as introduce some terminology. Consider a triple $R$, $R_1$, $p$, where $R$, $R_1$ are 
operators on a Hilbert space $\cH$ and $p\in\cH$ (we call this a triple on $\cH$). We say that two 
triples $R$, $R_1$, $p$ and $\wt R$, $\wt R_1$, $\wt p$ on $\cH$ and $\wt \cH$ respectively are 
unitary equivalent if there exists a unitary operator $U:\cH\to \wt\cH$ such that 
\begin{align*}
\wt R = U R U^*, \qquad \wt R_1 = U R_1 U^*, \qquad \wt p = Up.     
\end{align*}

%
%


 
The main result, Theorem \ref{t:main 02} easily follows, see Section \ref{s: Borg to main}  from 
the theorem below, applied to the operators $W=R^2$ and $W_1=R_1^2 = W-pp^*$. 

\begin{thm}[Abstract Borg's Theorem]
\label{t:Borg 01}
Given two sequences $(\lambda_k^2)_{k\ge 1}$ and $(\mu_k^2)_{k\ge 1}$ satisfying intertwining 
relations \eqref{intertwine 01} and such that $\lambda_k^2\to 0$ as $k\to \infty$, there exist a 
unique \tup{(}up to unitary equivalence\tup{)} triple $W$, $W_1$, $p$, such that  
\begin{enumerate}
\item  $W=W^*\ge 0$, $\Ker W=\{0\}$ is a 
compact operator with simple eigenvalues $( \lambda_k^2)_{k=1}^{\infty}$;
\item $p\in\cH$ and $W_1 = W-pp^*$ is a compact operator with non-zero eigenvalues $(\mu_{k}^2 
)_{k=1}^{\infty}$  
\tup{(}$W_1$ can 
also have a simple eigenvalue at $0$, and it is not hard to show that all the eigenvalues are 
simple\tup{)}.
\end{enumerate}

Moreover, $\|W^{-1/2} p\|=1$ if and only if 
\begin{align}
\label{e: norm q = 1 02}
\sum\limits_{j=1}^{\infty} \Bigg( 1-\frac{\mu_{j}^{2}}{\lambda_{j}^{2}} \Bigg) = \infty;
\end{align}
in addition, if \eqref{e: norm q = 1 02} holds, then  $\|W^{-1} p\|=\infty$ (i.e.~$p\notin\ran W$) 
if and only if 
\begin{align}
\label{e: q not in ran R 02}
\sum_{j=1}^\infty \left( \frac{\mu_j^2}{\lambda_{j+1}^2} -1 \right) =\infty .
\end{align}
\end{thm}

\begin{rem}
\label{r:abstract Borg}
The original Borg's  theorem \cite{Borg_Acta_1946} states that the potential $q$ of a 
Schr\"{o}dinger operator $L$,  $Ly = y'' + q(x)y$ on an interval is uniquely defined by the two 
sets of eigenvalues, corresponding to two  specific boundary conditions. Later Levinson 
\cite{Levinson 1949} extended this result by showing that essentially any non-degenerate pair of 
self-adjoint boundary conditions would work. 

Changing boundary conditions for a Schr\"{o}dinger operator is essentially a rank one perturbation 
(by an unbounded operator).  Namely, if $L_1$ and $L_2$ are Schr\"{o}dinger operators on an 
interval with the same potential, but with two different self-adjoint boundary conditions, then for 
any $\lambda\notin \sigma (L_1) \cup\sigma(L_2)$ the difference $(L_1-\lambda I)^{-1} - 
(L_2-\lambda I)^{-1}$ is a rank one operator (and the operators $(L_1-\lambda I)^{-1}$, 
$(L_2-\lambda I)^{-1}$ are compact). Thus, by picking a real $\lambda$ the problem can be reduced 
to rank one perturbations of compact self-adjoint operators. 

Our Theorem \ref{t:Borg 01} deals with rank one perturbations of (abstract) compact self-adjoint 
operators, hence the name. We do not  assume that our operators came from Schr\"{o}dinger 
operators, so we only reconstructing the spectral measure, and are not concerned with the 
reconstruction of the potential.  However, it is well known how to  reconstruct the potential from 
the spectral measure, or, more precisely, from the Titchmarsh--Weyl $m$-function, so it should be 
possible to get the Borg's result from our abstract theorem. 

Note also, that Theorem \ref{t:Borg 01} give not only uniqueness, but the existence as well.
\end{rem}

\subsection{Abstract Borg's Theorem implies the main result} 
\label{s: Borg to main} Let us now explain how the above Theorem \ref{t:Borg 01} implies Theorem 
\ref{t:main 02}. First of all, if we know  the triple $W=R^2$, $W_1=R_1^2$, $p$, and know the 
eigenvalues of $R$ and $R_1$ (it is sufficient to know only their signs), we can reconstruct the 
unique triple $R$, $R_1$, $p$ just by taking appropriate square roots of $W$ and $W_1$. Namely, if 
$u_k$ and $v_k$ are eigenvectors of $W$ and $W_1$, 
\begin{align*}
W u_k = \lambda_k^2 u_k, \qquad W_1 v_k = \mu_k^2 v_k, 
\end{align*}
we put 
\begin{align*}
R u_k = \lambda_k u_k, \qquad R_1 v_k = \mu_k v_k
\end{align*}
(and of course, we put $R_1 x = 0$ for $x\in \ker W_1$). And it is easy to see that a self-adjoint 
square root with prescribed signs of eigenvalues is unique and is given by the above formulas.

So, if we are given the eigenvalues $(\lambda_k)_{k\ge1}$,  $(\mu_k)_{k\ge1}$, we first construct 
the (unique up to unitary equivalence)  triple $W$, $W_1$, $p$ using the eigenvalues 
$(\lambda_k^2)_{k\ge1}$,  $(\mu_k^2)_{k\ge1}$, and then take the (unique) square roots with 
prescribed signs of eigenvalues to get the operators $R$ and $R_1$. Note, that the 
triple $R$, $R_1$, $p$ is unique up to unitary equivalence. 

By Proposition \ref{p:asy stability} the operator $\Sigma^*=R_1 R^{-1}$ is asymptotically stable, 
so by Proposition \ref{uniqueHankel} there exists a unique Hankel operator $\Gamma$ such that the 
triple $\Gamma|_{(\ker\Gamma)^\perp}$, $(\Gamma S)|_{(\ker\Gamma)^\perp}$, 
$u:= \Gamma e_0 = \Gamma P\ci{(\ker\Gamma)^\perp} e_0$ is unitarily equivalent to the triple $R$, 
$R_1$, $p$. This implies, in particular, that the non-zero eigenvalues of $\Gamma$ and of $\Gamma 
S$ are simple and coinside with $\{\lambda_k\}_{k\ge1}$ and $\{\mu_k\}_{k\ge1}$ respectively. So, 
the existence and uniqueness part of 
Theorem \ref{t:main 02} is proved. 

As for the conditions for $\ker\Gamma=\{0\}$, we just note that the conditions $\|q\|=1$ and 
$q\notin\ran R$ from Proposition \ref{uniqueHankel} translate to $\|R^{-1}p\|=1$ and 
$\|R^{-2}p\|=\infty$. But $\|R^{-1}p\|=\|W^{-1/2}p\|$, $\|R^{-2}p\|=\|W^{-1}p\|$, and thus the 
statement 
about triviality of the kernel of $\Gamma$ follows. \hfill\qed

\subsection{Proof of the abstract Borg's Theorem: existence and uniqueness part}
First of all notice that the intertwining condition \eqref{intertwine 01} implies that the vector 
$p$ must be cyclic for $W$. Since everything is defined up to unitary equivalence, we than can 
assume without loss of generality that $W$ is the multiplication $M_s$ by the independent variable 
$s$ in the weighted space $L^2(\rho)$, where $\rho$ is the spectral measure, corresponding to the 
vector $p$. 

Recall, that this spectral measure can be defined as the unique (finite,  Borel, compactly 
supported) measure on $\R$ such that 
\begin{align*}
((W- zI)^{-1}p, p) = \int_\R \frac{d\rho(s)}{s-z}\qquad \forall z\in \C\setminus \R. 
\end{align*}
Since $W$ is a compact operator with eigenvalues $(\lambda_k^2)_{k\ge1}$, the measure $\rho$ is purely atomic, 
\begin{align}
\label{e: spectral measure rho}
\rho = \sum_{k\ge 1} a_k \delta_{\lambda_k^2}, \qquad a_k>0. 
\end{align}
Note also that in this representation the vector $p$ is represented by the function $1$ in $L^2(\rho)$.

Since everything is considered up to unitary equivalence, we can always assume that $W$ is the 
multiplication operator $M_s$ by the independent variable $s$ in the weighted space $L^2(\rho)$, 
where the spectral measure $\rho$ is given by \eqref{e: spectral measure rho}, with $p\equiv 1$. 
Note that in this representation position $\lambda_k^2$ of delta functions are fixed (because they 
must correspond to the eigenvalues of $W$), but the weights $a_k$ are at the moment unknown. 

Note, that for $p\equiv1$ a choice of the weights $a_k$ completely  defines the pair $W$, $p$ up to 
unitary equivalence.%
\footnote{Of course, any choice of non-vanishing weights $a_k>0$ gives unitary equivalent operators 
$W$, but the vectors $p\equiv1$ are not transformed according to the unitary equivalence. }
Moreover, the choice of the weights $a_k$ is  completely defines the triple $W$, $W_1$, $p$ (up to 
unitary equivalence). Indeed if the weights $a_k$ are known 
(recall that in our representation  $p\equiv1$), 
the operator $W_1$ is uniquely defined in the above spectral representation of $W$ as  $W_1 = W 
-pp^*$.  

\subsubsection{Cauchy transforms of spectral measures}
\label{s: C rho}
Let us recall some standard facts in perturbation theory. Let $W=W^*$ be a (bounded, for 
simplicity) self-adjoint operator, and let $p$ be its cyclic vector.  The spectral measure $\rho$ 
corresponding to the vector $p$ is defined  as the unique Borel measure on $R$ such that 
\begin{align}
\label{e: spectral measure 02}
F(z):= ((W- zI)^{-1}p, p) = \int_\R \frac{d\rho(s)}{s-z}\qquad \forall z\in \C\setminus \R. 
\end{align}
As it is customary in perturbation theory, define a family of rank one perturbations
\begin{align*}
W^{\{\alpha\}} := W+ \alpha p p^*, \qquad \alpha\in\R . 
\end{align*}
Note that in this notation $W_1 = W^{\{-1\}} $. 

The spectral measures $\rho^{\{\alpha\}}$  are defined as the unique Borel measures such that 
\begin{align}
\label{e: spectral measure 03}
F^{\{\alpha\}}(z):= ((W^{\{\alpha\}}- zI)^{-1}p, p) = \int_\R 
\frac{d\rho^{\{\alpha\}}(s)}{s-z}\qquad \forall z\in \C\setminus \R; 
\end{align}
the functions $F^{\{\alpha\}}$ are the Cauchy transforms of $\rho^{\{\alpha\}}$.

The relation between the Cauchy transforms $F$ and $F^{\{\alpha\}}$ 
is given by the famous Aronszajn--Krein formula,
\begin{align}
\label{e: F F^gamma}
F^{\{\alpha\}}= \frac{F}{1+\alpha F} , 
\end{align}
which is an easy corollary of the standard resolvent identities. 

Recall that our operator $W_1$ is exactly the operator $W^{\{-1\}}$, so let us for the consistency 
denote $F_1:= F^{\{-1\}}$. Identity \eqref{e: F F^gamma} can then be rewritten as 
\begin{align}
\label{e:F_1 via F}
F_1 & = \frac{F}{1-F}  
\intertext{so}
\label{e: F/F_1}
\frac{F}{F_1} &= 1-F. 
\end{align}

So, to prove the first part of Theorem \ref{t:Borg 01} (existence and uniqueness of the triple $W$, 
$W_1$, $p$), it is sufficient to show that there exists a unique measure $\rho$ of form \eqref{e: 
spectral measure rho}
such that 
if $F$ is the Cauchy transform of $\rho$
\begin{align}
\label{e: C rho}
F(z) = \int_\R \frac{d\rho(s)}{s-z}\qquad \forall z\in \C\setminus \R
\end{align}
then the function $F_1$ defined by \eqref{e:F_1 via F} has the poles exactly at points $\mu_k^2$. 
Here the points $\lambda_k^2$ and $\mu_k^2$ are given, and the weights $a_k>0$ are to be found.

\subsubsection{Guessing the function \texorpdfstring{$F$}{F}} We want to reconstruct the spectral 
measure $\rho$ from the sequences $\{\lambda_k\}_{k\ge1}$ and $\{\mu_k\}_{k\ge1}$. 

Denote  
\begin{align*}
\sigma:= \sigma(W) = \{\lambda_k^2 \}_{k\ge1} \cup\{0\}, \qquad
\sigma_1:= \sigma(W_1) = \{\mu_k^2 \}_{k\ge1} \cup\{0\}. 
\end{align*}
It trivial that $F$ and $F_1$ are analytic in $\C\setminus \sigma$ and $\C\setminus \sigma_1$ 
respectively, and having simple poles at points $\lambda_k$ and $\mu_k$, $k\ge 1$ respectively 
(note that while the measure $\rho_1= \rho^{\{-1\}}$ can have a mass at $0$, the point $0$ is not 
an isolated singularity). 

Identity \eqref{e: F/F_1}  (which holds on $\C\setminus(\sigma\cup\sigma_1)$) implies that $F/F_1$ 
has simple poles at the points $\lambda_k^2$, $k\ge 1$, and that it can be analytically extended to 
$\C\setminus\sigma$. The (isolated) zeroes of $1-F$ must be at points where $F(z)=0$, so they must 
be only at the poles of $F_1$, i.e.~at the points $\mu_k^2$, $k\ge1$. 

So the function $F/F_1=1-F$ must be a function analytic on $\C\setminus \sigma$ with simple poles 
at the points $\lambda_k^2$, $k\ge 1$ and simple zeroes at the points $\mu_k^2$. 

One can try to write such a function, namely one could guess that 
\begin{align*}
1-F(z) = \prod_{k\ge 1 } \left( \frac{z- \mu_k^2  }{z-\lambda_k^2 } \right) \,.
\end{align*}
We will show that this is indeed the case; from there we will trivially get the formula for $F$, 
and then compute the weights $a_k$.

\subsubsection{Existence and uniqueness of our guess for \texorpdfstring{$F$}{F}} 

Denote 
\begin{align}
\label{e:Phi}
\Phi(z):= \prod_{k\ge 0 } \left( \frac{z- \mu_k^2  }{z-\lambda_k^2 } \right) \,.
\end{align}
We have guessed that $1-F(z) = \Phi(z)$. But first we need to show that the product \eqref{e:Phi}
is well defined. 

\begin{lm}
\label{l:Phi converges}
The product \eqref{e:Phi} converges uniformly on compact subsets of $\C\setminus \sigma$, and 
moreover
\begin{align}
\label{e: Phi(infty)=1}
\Phi(\infty) := \lim_{z\to \infty} \Phi(z) = 1
\end{align}
\end{lm}
\begin{proof}
Fix a compact $K\subset \C\setminus\sigma$
Trivially 
\begin{align}
\label{e: 1 - Phi_k}
\left| 1-\frac{z- \mu_k^2  }{z-\lambda_k^2 } \right| = \left| \frac{\lambda_k^2 - \mu_k^2  
}{z-\lambda_k^2 } \right|
\le C(K) (\lambda_k^2 - \mu_k^2) \qquad \forall z\in K . 
\end{align}
Since
\begin{align}
\label{e: sum la_k - mu_k}
\sum_{k\ge1} (\lambda_k^2 - \mu_k^2) \le \sum_{k\ge1} (\lambda_k^2 - \lambda_{k+1}^2) \le 
\lambda_1^2 <\infty
\end{align}
we have 
\begin{align*}
\sum_{k\ge 1} \left| 1-\frac{z- \mu_k^2  }{z-\lambda_k^2 } \right| \le C(K) \lambda_1^2 <\infty 
\qquad \forall z\in\C\setminus\sigma. 
\end{align*}
But this implies that the product \eqref{e:Phi} converges uniformly on compact subsets of 
$\C\setminus \sigma$. 

To prove the second statement, take any $R>2\lambda_1^2$. Then,  see \eqref{e: 1 - Phi_k} 
\begin{align*}
\left| 1-\frac{z- \mu_k^2  }{z-\lambda_k^2 } \right| = \left| \frac{\lambda_k^2 - \mu_k^2  
}{z-\lambda_k^2 } \right|
\le \frac2R (\lambda_k^2 - \mu_k^2) \qquad \forall z:\ |z|>R. 
\end{align*}
Using \eqref{e: sum la_k - mu_k} we get that for all $|z|>R$
\begin{align*}
\sum_{k\ge 1} \left| 1-\frac{z- \mu_k^2  }{z-\lambda_k^2 } \right| \le \frac{2\lambda_1^2}{R} , 
\end{align*}
which immediately implies \eqref{e: Phi(infty)=1}. 
\end{proof}

We will need the following definition

\begin{df}
	Recall that an analytic in the upper half-place $\C_+$ function $f$ is called Herglotz (or 
	Nevanlinna) if $\im f(z)\ge 0$ for all $z\in\C_+$. 
	
	We should mention, that any non-constant Herglotz function satisfies the strict inequality 
	$f(z)>0$ for all $z\in\C_+$. 
\end{df}

\begin{lm}
\label{l: properties Phi}
The function $\Phi$ defined by \eqref{e:Phi} satisfies the following properties
\begin{enumerate}
\item $\Phi(\overline z) = \overline{\Phi(z)}$; in particular, $\Phi(x)$ is real for all 
$x\in\R\setminus\sigma$. 

\item The function $\Phi$ has simple poles at points $\lambda_k^2$, and its only zeroes are the 
simple zeroes at points $\mu_k^2$. 

\item $\displaystyle \lim_{z\to \infty} \Phi(z) = 1$. 

\item The function $-\Phi$ is a Herglotz (Nevanlina) function in $\C_+$, i.e.~$\im \Phi(z) <0$ for 
all $z\in\C_+$.  
\end{enumerate}
\end{lm}

\begin{proof}
Statements \cond1, \cond2  are trivial, statement \cond3 was stated and proved in Lemma \ref{l:Phi 
converges} above. 

Let us prove statement \cond4. Let $\Arg$ denote the \emph{principal value} of the argument, taking 
values in the interval $(-\pi, \pi]$. It is easy to see that for $z\in \C_+$
\begin{align*}
  - \Arg \left( \frac{z- \mu_k^2  }{z-\lambda_k^2 } \right)
= - \Arg \left( \frac{\mu_k^2 -z }{\lambda_k^2 - z} \right) =\alpha_k(z)
\end{align*}
where $\alpha_k(z) > 0$ is the aperture of the angle at which an observer at the point $z\in\C_+$ 
sees the interval $[\mu_k^2, \lambda_k^2]$. Since for the whole real line $\R$ the aperture is 
$\pi$ and the intervals $[\mu_k^2, \lambda_k^2]$ do not intersecr $\R_-=(-\infty, 0)$,  we can 
conclude that
\begin{align*}
0 < -\sum_{k\ge1} \Arg \left( \frac{\mu_k^2 -z }{\lambda_k^2 - z} \right) < \pi , 
\end{align*}
so $-\im \Phi(z)>0$ for all $z\in\C_+$. 
\end{proof}

\begin{lm}
\label{l:Phi unique} The function $\Phi$ defined by \eqref{e:Phi} is the only analytic in 
$\C\setminus\sigma$ function satisfying properties \cond1--\cond4 of Lemma \ref{l: properties Phi}. 
\end{lm}

\begin{proof}
Let $\Phi_1$ be another such function. 
Both functions have simple poles at $\lambda_k^2$ and simple zeros at $\mu_k^2$ (and these are the 
only isolated singularities and zeroes for both functions),  hence $\Psi(z):=\Phi(z)/\Phi_1(z)$ is 
analytic and zero-free in $\C\setminus \{0\}$.

Moreover, for $x\in \R\setminus\{0\}$ we have $\Psi(x)>0$. Indeed, on $\R\setminus 
\sigma\setminus\sigma_1$ the functions $\Phi_{1,2}$ are real and have the same sign, so $\Psi(x)>0$ 
on $\R\setminus \sigma\setminus\sigma_1$. Since $\Psi$ is continuous and zero-free on $\R \setminus 
\{0\}$, this tells us that $\Psi$ is positive on $\R \setminus \{0\}$. 

Next, let us notice that $\Psi(z)$ does not take real negative values. If $\im z>0$, then 
$\im\Phi(z)>0$, $\im\Phi_{1}(z)>0$, so $\Psi(z)=\Phi(z)/\Phi_1(z)$ cannot be negative real. If $\im 
z<0$, the symmetry $\Psi(\overline{z})= \overline{\Psi(z)}$ implies the same conclusion. And, as we 
just discussed above,  on the real line $\Psi$ takes positive real values. 

So, $\Psi$ omits infinitely many points, therefore by the Picard's Theorem the point $0$ is not an 
essential singularity for $\Psi$. Trivial analysis shows that $0$ cannot be a pole, otherwise 
$1/\Psi$ is analytic at 0, which also contradicts to the fact that $\Psi$ can't take negative real 
values. Hence the point $0$ is a removable singularity for function $\Psi$, so $\Psi$ is an entire 
function. 
By Liouville's Theorem, condition $\Psi(\infty) = 1$ implies that $\Psi\equiv 1$ for all $z \in 
\C$, so $\Psi\equiv\Psi_1$

We can avoid using Picard's theorem by considering the square root $\Psi^{1/2}$, where we take the 
principal branch of the square root (cut along the negative half-axis). Since $\Psi$ does not take 
negative real values, the function $\Psi^{1/2}$ is well defined (and analytic) on $\C\setminus 
\{0\}$.  Trivially $\re\Psi(z)^{1/2}\ge 0$, so by the Casorati--Weierstrass Theorem $0$ cannot be 
the essential singularity for $\Psi^{1/2}$. Again, trivial reasoning shows that $0$ cannot be a 
pole, so again, $\Psi^{1/2}$ is an entire function. The condition $\Psi^{1/2}(\infty)=1 $ then 
implies that $\Psi^{1/2}(z)\equiv 1$.
\end{proof}

\subsubsection{Computing \texorpdfstring{$F$}{F} and the spectral measure 
\texorpdfstring{$\rho$}{rho}}
Now the proof of the first part of Theorem \ref{t:Borg 01}  (existence and uniqueness of the triple 
$W$, $W_1$, $p$) is almost completed. Namely, we define $F:= 1-\Phi$, where $\Phi$ is defined by 
\eqref{e:Phi}. 
The function $F_1$ defined by \eqref{e: F F^gamma} has poles exactly at the points $\mu_k^2$. 
Therefore, if we show that $F$ is the Cauchy transform \eqref{e: C rho} of the measure $\rho$ of 
form \eqref{e: spectral measure rho}, then according to the discussion at the end of Section 
\ref{s: C rho} we get the existence of the triple $W$, $W_1$, $p$.  

The uniqueness of the triple $W$, $W_1$, $p$, i.e.~the uniqueness of the measure $\rho$ follows 
from the uniqueness of the function $\Phi$, see Lemma \ref{l:Phi unique}. Namely, if $F$ is the 
Cauchy transform \eqref{e: C rho} of the measure $\rho$ of form \eqref{e: spectral measure rho}, 
and $F_1$ given by \eqref{e: F F^gamma} has poles exactly at the points $\mu_k^2$,  then the 
function $\Phi:=1-F$  satisfies the properties \cond1--\cond4 of Lemma \ref{l: properties Phi}. But 
by Lemma \ref{l:Phi unique} such function $\Phi$ is unique, and so are the function $F$ and so the 
measure $\rho$. 

The fact $F$ is indeed the Cauchy transform of an appropriate measure $\rho$ can be obtained from 
the general theory of Herglotz functions.%
\footnote{a computation will be needed to show that $\rho$ does not have a mass at $0$.} 
However,  we will not use a general theory, but present an elementary proof,  see the lemma below.

%
\begin{align*}
F(z )=\int_\R \frac{d \rho(s)}{s-z}, 
\end{align*}
and that the measure $\rho$ is supported on $\sigma$.   
\begin{lm}
\label{l: Phi decomposition}
The function $\Phi$ defined by \eqref{e:Phi} can be decomposed as
\begin{align}
\label{e: Phi 03}
\Phi(z) = 1 -  \sum_{n\ge 1} \frac{a_n}{\lambda_n^2-z} , 
\end{align}
where 
\begin{align}
\label{e: a_n}
a_n = (\lambda_n^2 - \mu_n^2 ) \prod_{k\ne n} \left(  \frac{\lambda_n^2 - \mu_k^2}{\lambda_n^2 - 
\lambda_k^2}     \right)
\end{align}
\end{lm}

\begin{proof}
Consider functions $\Phi\ci N$
\begin{align*}
\Phi\ci N(z) = \prod_{k=1}^N \left( \frac{z- \mu_k^2  }{z-\lambda_k^2 } \right)
\end{align*}
Trivially
\begin{align}
\label{e: Phi_N}
\Phi\ci N(z) = 1 -  \sum_{n\ge 1} \frac{a_n^N}{\lambda_n^2-z} 
\end{align}
where 
\begin{align*}
a_n^N = (\lambda_n^2-\mu_n^2) \prod_{\substack{k=1 \\ k\ne n}}^N \left(  \frac{\lambda_n^2 - 
\mu_k^2}{\lambda_n^2 - \lambda_k^2}     \right)  \qquad \text{if } n\le N, 
\end{align*}
and $a_n^N=0$ if $n>N$. 

 We know, see Lemma \ref{l:Phi converges} that $\Phi_{N}(z)$ converges   to $\Phi(z)$ uniformly on 
 compact subsets of $  \C \setminus \sigma$.
%
Hence, to prove the lemma it remains to show that 
\begin{align*}
\sum\limits_{n= 1}^N \frac{a_n^N}{\lambda_n^2-z} \to \sum\limits_{n\ge 1} \frac{a_n}{\lambda_n^2-z} 
\qquad \text{as }N\to\infty
\end{align*}
uniformly on compact subsets of $\C\setminus \sigma$. 

Take $z = 0$ in \eqref{e: Phi_N}. Then
\begin{align*}
1-\sum\limits_{n \geq 1}\frac{a_n^N}{\lambda_{n}^2}=\prod\limits_{k=1}^{N}\left( \frac{ \mu_k^2  
}{\lambda_k^2 } \right) >0, 
\end{align*}
so $\sum\limits_{n \geq 1}\frac{a_n^N}{\lambda_{n}^2} \le 1$.

Notice that for any fixed $n$ the sequence $a_n^N \nearrow a_n$  as $N\to\infty$, so 
$\sum\limits_{n \geq 1}\frac{a_n}{\lambda_n^2} \le 1$. 

Take an arbitrary compact $K\subset \C\setminus \sigma$. Clearly for any $z\in K$
\[
\left| \frac{a_n^N}{\lambda_n^2 - z} \right| \le \frac{a_n^N}{\dist(K, \sigma)} \le 
\frac{a_n}{\dist(K, \sigma)} 
\le \frac{\lambda_1^2}{\dist(K, \sigma)} \cdot \frac{a_n}{\lambda_n^2}, 
\]
so the condition $\sum_{n\ge 1} a_n/\lambda_n^2 \le 1$ implies that 
the series $\sum\limits_{n\ge 1} \frac{a_n^N}{\lambda_n^2-z}$ converges uniformly on the compact  
$K$. 
\end{proof}

\subsection{Proof of the abstract Borg's theorem: the trivial kernel condition}
\label{s: trivial kernel condition}

To complete the proof of Theorem \ref{t:Borg 01} is remains to show that the condition $\|W^{-1/2} 
p\|=1$ is equivalent to \eqref{e: norm q = 1 02}, and that if \eqref{e: norm q = 1 02} holds, then 
the condition $\|W^{-1} p\|=\infty$ is equivalent to \eqref{e: q not in ran R 02}. 

Let us first investigate the condition $\|W^{-1/2} p\|=1$. The operator $W$ is the multiplication 
by the independent variable $s$ in the weighted space $L^2(\rho)$, where $\rho= \sum_{k\ge1} a_n 
\delta_{\lambda_n^2}$ with $a_n$ given by 
\eqref{e: a_n}. 
Recall, see Lemma \ref{l: Phi decomposition}, that 
\begin{align}
\label{e: Phi decomp 01}
\prod\limits_{k=1}^{\infty}\Bigg( \frac{z-\mu_{k}^{2}}{z-\lambda_{k}^{2}} \Bigg)
= 1 - \sum\limits_{k=1}^{\infty}\frac{a_k}{\lambda_{k}^{2}-z}
\end{align}  
Plugging the real $x<0$ into \eqref{e: Phi decomp 01} and 
taking the limit of both sides as $x\to 0^-$ we get that   
\begin{align*}
\prod\limits_{k=1}^{\infty}\frac{\mu_{k}^{2}}{\lambda_{k}^{2}}  = 
1-\sum\limits_{k=1}^{\infty}\frac{a_{k}}{\lambda_k^2};  
\end{align*}
the interchange of limit and sum (product) can be justified, for example, by the monotone 
convergence theorem. 

Therefore $\sum_{k\ge 1} a_k/\lambda_k^2 = 1$ if and only if $\prod_{k\ge0} 
(\mu_k^2/\lambda_k^2)=0$. The latter condition is equivalent to 
\begin{align*}
\sum_{k\ge1} \left( 1- \frac{\mu_k^2}{\lambda_k^2}  \right) =\infty,
\end{align*}
which  is exactly the condition \eqref{e: norm q = 1 02}.

Now, let us investigate the condition $\|W^{-1} p\|=\infty$. It can be rewritten as 
\begin{align}
\label{e: sum a lambda^4}
\sum_{k\ge 1} \frac{a_k}{\lambda_k^4} = \infty.
\end{align}
Rewriting identity \eqref{e: Phi decomp 01} as
\begin{align*}
\prod_{k=1}^{\infty} \left( \frac{z- \mu_k^2  }{z-\lambda_k^2 } \right) &= 
1-\sum\limits_{k=1}^{\infty}a_{k}\frac{(\lambda_{k}^2-z)+z}{\lambda_{k}^2(\lambda_{k}^2-z)} \\
&= 1-\sum\limits_{k \geq 1}\frac{a_k}{\lambda_k^2} - \sum\limits_{k \geq 1}\frac{a_k z}{\lambda_k^2 
(\lambda_k^2 -z)} ,
\end{align*}
we see that if the condition \eqref{e: norm q = 1 02} holds, then 
\begin{align*}
-\frac{1}{z}\prod\limits_{k=1}^{\infty}\Bigg( \frac{z-\mu_{k}^{2}}{z-\lambda_{k}^{2}} 
\Bigg)=\sum\limits_{k=1}^{\infty}\frac{a_k}{\lambda_k^2(\lambda_{k}^{2}-z)} .  
\end{align*}

Substituting $z=-\lambda_N^2$ we get  that 
\begin{align*}
\frac{1}{\lambda_N^2}\prod\limits_{k=1}^{\infty}\Bigg( \frac{\lambda_N^2 + \mu_{k}^{2}}{\lambda_N^2 
+\lambda_{k}^{2}} \Bigg)=\sum\limits_{k=1}^{\infty}\frac{a_k}{\lambda_k^2(\lambda_N^2 
+\lambda_{k}^{2})} .
\end{align*}
By the monotone convergence theorem 
\begin{align*}
\lim_{N\to\infty} \sum\limits_{k\ge 1}\frac{a_k}{\lambda_k^2(\lambda_N^2 +\lambda_{k}^{2})} =
\sum\limits_{k\ge 1}\frac{a_k}{\lambda_k^4} , 
\end{align*}
so the condition \eqref{e: sum a lambda^4} can be rewritten as 
\begin{align}
\label{e: prod1 = infty}
\lim_{N\to\infty} \frac{1}{\lambda_N^2}\prod\limits_{k=1}^{\infty}
\Bigg( \frac{\lambda_N^2 + \mu_{k}^{2}}{\lambda_N^2 +\lambda_{k}^{2}} \Bigg) =\infty
\end{align}
The following simple lemma completes the proof. 

\begin{lm}
\label{l: prod = infy 02}
The condition \eqref{e: prod1 = infty} holds if and only if 
\begin{align*}
\sum_{k\ge 1}\left(\frac{\mu_k^2}{\lambda_{k+1}^2} -1\right) = \infty .
\end{align*}
\end{lm}

\begin{proof}
First of all notice that 
\begin{align}
\label{e: prod bounds 01}
0 < C \leq \prod\limits_{k=N}^{\infty}\Bigg( 
\frac{\lambda_{N}^{2}+\mu_{k}^{2}}{\lambda_{N}^{2}+\lambda_{k}^{2}} \Bigg) \leq 1 
\end{align}
with $C>0$ independent of $N$. 

Indeed, for all $k\ge N$ we trivially have
\begin{align}
\label{e: term bound 01}
\frac12 \le \frac{\lambda_{N}^{2}+\mu_{k}^{2}}{\lambda_{N}^{2}+\lambda_{k}^{2}} \le 1 .
\end{align}
The upper bound in \eqref{e: term bound 01} trivially implies the upper bound in \eqref{e: prod 
bounds 01}. 

To get the lower bound in \eqref{e: prod bounds 01} we use the estimate \eqref{e: term bound 01} 
and the inequality
\begin{align*}
\ln x \ge (\ln2) (x-1), \qquad \forall\  x \in [1/2, 1]. 
\end{align*}
Thus 
\begin{align*}
\sum_{k=N}^\infty \ln2 \left( \frac{\lambda_{N}^{2}+\mu_{k}^{2}}{\lambda_{N}^{2}+\lambda_{k}^{2}} 
-1   \right) 
= \sum_{k=N}^\infty  -\ln2  \frac{\lambda_{k}^{2} - \mu_{k}^{2}}{\lambda_{N}^{2}+\lambda_{k}^{2}}
\ge -\ln2 \sum_{k=N}^{\infty} \frac{\lambda_k^2-\mu_k^2}{\lambda_N^2}
\ge -\ln2
\end{align*}
we see that the lower bound in \eqref{e: prod bounds 01} holds with $C_1=\frac{1}{2}$. 

The estimate \eqref{e: prod bounds 01} implies that the condition \eqref{e: prod1 = infty} is 
equivalent to 
\begin{align*}
\lim_{N\to\infty} \frac{1}{\lambda_N^2}\prod\limits_{k=1}^{N-1}
\Bigg( \frac{\lambda_N^2 + \mu_{k}^{2}}{\lambda_N^2 +\lambda_{k}^{2}} \Bigg) =\infty.  
\end{align*}
Since $\lambda_N^2 \le \lambda_N^2 + \lambda_k^2 \le 2 \lambda_N^2$ for $k\ge N$, the above 
condition is equivalent to 
\begin{align}
\label{e: prod1 = infty 02}
\lim_{N\to\infty} 
\prod\limits_{k=1}^{N-1}
\Bigg( \frac{\lambda_N^2 + \mu_{k}^{2}}{\lambda_N^2 +\lambda_{k+1}^{2}} \Bigg) =\infty, 
\end{align}
Denote
\begin{align*}
G_{N}(z) := 
\prod\limits_{k=1}^{N-1}\Bigg( \frac{\mu_{k}^{2} - z}{\lambda_{k+1}^{2} - z} \Bigg)
\end{align*}
The function $G_N$ is analytic in the half-plane $\re z < \lambda_{N}^2$ and satisfies the 
inequality $|G_N(z)|\ge 1$ there. Therefore the function $\ln|G_N|$ is harmonic and non-negative in 
the disc $D_N$ of radius $2\lambda_N^2$ centered at $-\lambda_N^2$.  So by the Harnack inequality 
\begin{align*}
\frac13 \ln |G_N(-\lambda_N^2)| \le \ln |G_N(0)| \le 3  \ln |G_N(-\lambda_N^2)| .
\end{align*}

Note that $G_N(0) = \prod_{k=1}^{N-1} \mu_k^2/\lambda_{k+1}^2 $, so the condition \eqref{e: prod1 = 
infty 02} (which translates to the condition $\lim_{N\to\infty} G\ci N (-\lambda_N^2) =\infty$)  is 
equivalent to 
\begin{align*}
\lim_{N\to\infty} G\ci N(0) =\lim_{N\to\infty} \prod_{k=1}^{N-1} \frac{\mu_k^2}{\lambda_{k+1}^2} 
=\infty.
\end{align*}
The latter condition is equivalent to the condition 
\begin{align*}
\prod_{k\ge 1} \frac{\mu_k^2}{\lambda_{k+1}^2} = \infty, 
\end{align*}
which in turn is equivalent to \eqref{e: q not in ran R 02}. 
\end{proof}

\def\cprime{$'$}
  \def\lfhook#1{\setbox0=\hbox{#1}{\ooalign{\hidewidth\lower1.5ex\hbox{'}\hidewidth\crcr\unhbox0}}}
\providecommand{\bysame}{\leavevmode\hbox to3em{\hrulefill}\thinspace}
\providecommand{\MR}{\relax\ifhmode\unskip\space\fi MR }
\providecommand{\MRhref}[2]{%
  \href{http://www.ams.org/mathscinet-getitem?mr=#1}{#2}
}


\begin{thebibliography}{10}

\bibitem{Birm-Sol_book}
M.~S.~Birman, M.~Z.~Solomjak, \emph{Spectral Theory of Self-Adjoint Operators in Hilbert space,} 
translated from the Russian by S.~Khrushchev and V.~Peller, Springer-Verlag, Dordrecht Netherlands, 
1987.  

\bibitem{Borg_Acta_1946}
G.~Borg, \emph{Eine Umkehrung der Sturm-Liouvilleschen Eigenwertaufgabe : Bestimmung der 
Differentialgleichung durch die Eigenwerte}, Acta Math. 78(1946) 1-96

\bibitem{Gerard 2010}
P.~Gerard, S.~Grellier, \emph{The cubic Szeg\"o equation}, Ann. Scient. E\'c. Norm. Sup. 43(2010), 
761-810

\bibitem{Gerard 2012}
P.~Gerard, S.~Grellier, \emph{Invariant Tori for the cubic Szeg\"o equation}, Invent. Math. 
187(2012), 707--754

\bibitem{Ger-Grill_ISP_2014}
P.~G\'erard, S.~Grellier, \emph{Inverse spectral problems for compact Hankel Operators}, 
J.~Inst.~Math.~Jessieu \textbf{13}(2014), no.~2, 273--301.

\bibitem{GG-Ast} 
{ P.~G\'erard, S.~Grellier, }
\emph{The cubic Szeg\H{o} equation and Hankel operators}, 
 {\it Ast\'erisque} \textbf{389} (2017).
 
\bibitem{Levinson 1949}
N.~Levinson, \emph{The inverse Sturm-Liouville problem}, Matematisk Tidsskrift.~B(1949), pp.25--30. 

\bibitem{MPT-Acta}
A.~V. Megretski{\u\i}, V.~V. Peller, and S.~R. Treil{\cprime}, \emph{The
	inverse spectral problem for self-adjoint Hankel operators}, Acta Math.
\textbf{174} (1995), no.~2, 241--309. 

\bibitem{Nikolski1}
N.~Nikolski, \emph{Operators, functions, and systems: an easy reading.
	{V}ol. 1,  Hardy, Hankel, and Toeplitz}, Mathematical Surveys and Monographs, vol.~92, American
Mathematical Society, Providence, RI, 2002, Translated from the French by A.~Hartmann and revised 
by the author.

\bibitem{Nikolski2}
{ N.~K.~Nikolski,}
\emph{Operators, functions, and systems: an easy reading.
	{V}ol. 2, Model operators and systems}, Mathematical Surveys and Monographs, vol.~93, American
Mathematical Society, Providence, RI, 2002, 
Translated from the French by A.~Hartmann and revised by the author.

\bibitem{Harmonic analysis}
B.~Sz.-Nagy, C.~Foias, H.~Bercivici, and L.~K\'erchy, \emph{Harmonic Analysis of Operators on 
Hilbert Space. Second Edition}



\end{thebibliography}
 \end{document}